\documentclass[reqno,11pt]{amsart}
\usepackage[T1]{fontenc}
\usepackage{amsfonts, amsthm}
\usepackage{amssymb, amsmath}
\usepackage{times}
\usepackage{amsfonts,mathrsfs}
\usepackage{enumerate, graphicx, setspace, setspace}

\newcommand{\R}{\mathbb{R}}
\newcommand{\C}{\mathbb{C}}

\newcommand{\SO}{\mathop{\mathrm{SO}}\nolimits}
\newcommand{\SU}{\mathop{\mathrm{SU}}\nolimits}
\newcommand{\PSU}{\mathop{\mathrm{PSU}}\nolimits}

\newcommand{\U}{\mathop{\mathrm{U}}\nolimits}

\newcommand{\dee}{\mathop{\! \, \rm d \!}\nolimits}

\theoremstyle{plain}
\newtheorem{claim}{\sc Claim}[section]

\newtheorem{proposition}[claim]{\sc Proposition}
\newtheorem{theorem}[claim]{\sc Theorem}

\newtheorem{conjecture}[claim]{\sc Conjecture}

\theoremstyle{remark}

\theoremstyle{definition}
\newtheorem{definition}{\sc Definition}[section]

\numberwithin{equation}{section}

\begin{document}

\title{On curves of constant torsion I}

\author{Larry M. Bates and O. Michael Melko}

\date{\today}

\begin{abstract}
We give an explicit construction of a closed curve with constant torsion and everywhere positive curvature. We also discuss the restrictions on closed curves of constant torsion when they are constrained to lie on convex surfaces.
\end{abstract}

\maketitle


\section{Introduction}\label{scIntroduction}

\noindent In a remarkable note in 1887, {\sc Koenigs} \cite{koenigs} gave a geometrical argument that a spherical curve $B(s)$, parametrized by arclength $s$, could be used to produce a curve $\gamma(s)$ of constant torsion $\tau$ by means of the integral formula
\begin{equation}\label{eqIntegralRep}
\gamma(s) = \frac{1}{\tau} \int B \times \frac{dB}{ds} \, ds.
\end{equation}
This (as will be shown in the sequel) follows immediately from the Frenet-Serret formulas if one thinks of the curve $B$ as the binormal vector of \( \gamma \).  
Subsequently, various contemporary authors, notably {\sc fouch\'e} \cite{fouche}, {\sc lyon} \cite{lyon}, and {\sc fabry} \cite{fabry}, developed this representation (attributed in \cite{lyon} to Darboux) to study more fully curves of constant torsion, especially in the algebraic case.
This study was facilitated by the observation that the integral formula (\ref{eqDarbouxRep}) could be realized in terms of a trio of arbitrary functions $h$, $k$, and $l$ in a manner analogous to the Weierstrass representation for minimal surfaces in terms of a triplet of analytic functions whose squares sum to zero.
More precisely, the representation for a curve of constant torsion is given by
\begin{equation}\label{eqDarbouxRep}
\begin{aligned}
 x & = & \frac{1}{\tau} \int \frac{l\,dk - k\,dl}{h^2 +k^2+l^2}, \\
 y & = & \frac{1}{\tau} \int \frac{h\,dl - l\,dh}{h^2 +k^2+l^2}, \\
 z & = & \frac{1}{\tau} \int \frac{k\,dh - h\,dk}{h^2 +k^2+l^2}. \\
\end{aligned}
\end{equation}
These authors present numerous examples of algebraic curves of constant torsion, but, as far as we can tell, they do not consider the problem of finding curves of this type that are closed.
Examples of such curves were found much later by {\sc calini} and {\sc ivey} \cite{calini-ivey}, who applied the method of B\"{a}cklund transformations to a generalized class of closed elastica to find new examples of curves of constant torsion.
Their work exploits the intimate connection between curves of constant torsion and surfaces of constant negative Gaussian curvature and involves some rather sophisticated machinery.
What is interesting here is that the examples they give all have curvature that changes sign.
In a parallel line of research, several authors proved related versions of the four vertex theorem and showed, in particular, that there is no simple closed curve of constant nonzero torsion on a convex surface and no closed curve of constant nonzero torsion on a sphere.

It therefore came as a pleasant surprise to us when we found an explicit construction of an unknotted closed curve of constant torsion and everywhere positive curvature.
We turn to this construction after reviewing some basic properties of the Frenet equations.
This is followed by a discussion of curves of constant torsion on convex surfaces and further notes on the literature.
Related matters will be discussed in a companion paper \cite{bates-melko} in preparation by the authors.

\subsection{The Frenet equations, signed curvature, and the Darboux representation}

We review the Frenet equations (or Frenet-Serret formulas), first of all, to establish some notation conventions.
Furthermore, as many authors (e.g. {\sc Stoker} \cite{stoker}) claim that curves in \( \R^3 \) have positive curvature by definition, we use this opportunity to clarify what we mean by curvature that changes sign.

By means of parallel translation, we can identify a map 
$F : [0, L] \rightarrow \SO(3)$
from the closed interval $[0, L]$ into the special orthogonal group with a moving frame along $\gamma$.
We refer to such a map as a \emph{framing} of $\gamma$, and we say that the pair $(\gamma, F)$ is a \emph{framed curve}.
By convention, we write $F = (f_1, f_2, f_3)$, where $f_1$, $f_2$, and $f_3$ are the column vectors of $F$.
Suppose further that $A : [0, L] \rightarrow \mathfrak{so}(3)$ is a map from the interval $[0, L]$ into the Lie algebra $\mathfrak{so}(3)$ of skew-symmetric matrices.
The fundamental theorem of ordinary differential equations then guarantees that the system defined by the equations $\gamma^\prime = f_1$ and $F^{\prime} = F A$ has a solution that is unique up to a choice of initial conditions $\gamma(0) = \gamma_0$ and $F(0) = F_0$.
It follows immediately that $f_1 = T$ is the unit tangent field of $\gamma$ and that $\gamma$ is parametrized by arclength.
Suppose now that $\kappa$ and $\tau$ are continuous functions on $[0, L]$, which we assume to be at least piecewise continuously differentiable.
Then, in the special case that $A$ has the form given in (\ref{eqFrenetMatrixForm}), we say that $F$ is the \emph{Frenet frame} of $\gamma$, and that $f_2 = N$ and $f_3 = B$ are the \emph{principal normal} and \emph{binormal}, respectively, of the curve.
Furthermore, we refer to the system
\begin{equation}\label{eqFrenetMatrixForm}
\gamma^{\prime} = T, \qquad
F^{\prime} = F\,A, \qquad
A = \begin{pmatrix}
		0		&	-\kappa &	0 \\
		\kappa	&	0		&	-\tau \\
		0		&	\tau 	&	0
	\end{pmatrix}.
\end{equation}
as the \emph{Frenet equations} associated to $\gamma$.
Note that the Frenet equations are usually given as
\begin{equation}\label{eqFrenetStandard}
T^{\prime} = \kappa\,N, \qquad
T^{\prime} = -\kappa\,T + \tau\,B, \qquad
B^{\prime} = -\tau\,N.
\end{equation}
The matrix form of these equations is $F^{\prime} = F\,A$, which we have augmented with the \emph{tangent equation} $\gamma^{\prime} = T$ for convenience.

We recall that the pairs $(T, N)$, $(N, B)$, and $(B, T)$ span, respectively, the \emph{osculating plane}, the \emph{normal plane}, and the \emph{rectifying plane} of $\gamma$ for a given value of $s$.
The \emph{curvature} $\kappa$ measures the rate at which $\gamma$ bends in the osculating plane and the \emph{torsion} $\tau$ measures the rate at which the normal plane rotates about $T$.
We say that the curvature function $\kappa : [0,L] \rightarrow \R^3$ is \emph{admissible} if, in addition to the above continuity assumptions, it has only isolated zeros and if $\kappa^\prime(s)$ exists and is nonzero whenever $\kappa(s) = 0$.
We assume throughout that $\kappa$ is admissible.
The curve $\gamma$ is uniquely determined, up to translation, by the  Frenet frame $F$, and, conversely, $F$ is uniquely determined by $\gamma$ if the curvature $\kappa$ is admissible.
We will say that $\kappa$ is \emph{signed} if it is admissible and vanishes at least once in the open interval $(0, L)$.

The curvature of $\gamma$ is often defined as $\kappa = \|T^\prime\|$, where $\| \ \| $ denotes the Euclidean norm.
It is argued that, although signed curvature can be defined for plane curves, no intrinsic meaning can be given to the sign of curvature for curves in $\R^3$ since a plane curve in $\R^3$ can be mapped by a rigid motion onto a congruent curve with (planar) curvature of opposite sign.
We hold, however, that the \emph{transition} of curvature from positive to negative values (or vice-versa) is a real geometric property of a space curve: the introduction of signed curvature only weakens somewhat the uniqueness of $\kappa$ given $\gamma$ (up to rigid motion).
Providing $\gamma$ with a framing (that we always suppose to be consistent with the standard orientation of $\R^3$) removes this ambiguity.

Note that if $\kappa$ is signed, then $\kappa(s_0) = 0$ if and only if $\gamma$ passes through its rectifying plane $(T, B)$ at $\gamma(s_0)$.
Informally, a curve of constant torsion and slowly varying curvature will look roughly like a helix for which a directrix (i.e. an axis) can be defined.
A rapid change of sign in the curvature then results in a change of direction of the curve's directrix, causing it to move to the opposite side of the rectifying plane where $\kappa$ vanishes.
It was a long standing belief of one of the authors that this sort of behavior is necessary in order to obtain a closed curve of constant torsion.
We provide a counterexample to this hypothesis in the Section \ref{scSpace}.

It is sometimes useful to consider different framings of a curve, so we point out here how two framings are related.
Suppose that $\hat{F}$ is another framing of $\gamma$.
Then $\hat{F}$ is related to $F$ by the rule $\hat{F} = F\,g^{-1}$ for some map $g:[0, L] \rightarrow \SO(3)$.
One easily verifies that $\hat{F}^{\prime} = \hat{F}\,\hat{A}$, where $\hat{A}= -g\,g^{-1} + g\,A\,g^{-1}$, and that the tangent equation becomes $\gamma^{\prime} = \hat{f}_1\,g$.
This transformation can be thought of as a \emph{change of gauge}, where $g^{-1}$ is used instead of $g$ to make the map $F \rightarrow F\,g^{-1}$ a left action, and we see that it results in a system of equations equivalent to (\ref{eqFrenetMatrixForm}), giving rise to the same curve $\gamma$ as a solution.
Of particular interest is the Darboux frame, which is a framing that is adapted to a surface containing the curve.
Its precise definition and properties will be discussed in Section \ref{scOvaloids}.

To close this section, we indicate how the \emph{Darboux representation} (\ref{eqDarbouxRep}) for a curve of nonzero constant torsion is derived from (\ref{eqFrenetStandard}), and then we prove a simple, but fundamental, proposition.
Regarding the former, one simply substitutes $N = -B^{\prime} / \tau$ into the equation $T = N \times B$ and integrates to obtain (\ref{eqIntegralRep}).
Then (\ref{eqDarbouxRep}) is obtained by substituting $B = (h,k,l)/(h^2 + k^2 + l^2)$ for an arbitrary choice of functions $h$, $k$, and $l$.

\begin{proposition}\label{pKappaBgKappa}
Suppose that $B$ is a curve on the unit sphere parametrized by arclength, and suppose that $\gamma$ and $B$ are related by (\ref{eqIntegralRep}), where $\tau$ is a nonzero constant.
If $\kappa^B_g$ denotes the geodesic curvature of $B$, then $\kappa^B_g = \tau\,\kappa$, where $\kappa$ and $\tau$ denote the curvature and torsion of $\gamma$ respectively.
\end{proposition}

\begin{proof}
The unit tangent to $\gamma$ is given by $T = B \times B^\prime / \tau$.
It follows that $T\perp B$ and $T \perp B^\prime$, from which we conclude that $\kappa^B_g = \langle B^{\prime\prime}, T \rangle$, where $\langle\ ,\ \rangle$ denotes the Euclidean inner product on $\R^3$.
However, by (\ref{eqFrenetStandard}), we have that $B^{\prime\prime} = \left(-\tau\,N \right)^\prime = \kappa\,\tau\,T - \tau^2\, B$, and hence that $\kappa^B_g =  \langle\kappa\,\tau\,T + \tau^2\, B, T \rangle = \kappa\,\tau$.
\end{proof}

\section{Curves of constant torsion in space}\label{scSpace}

In this section, we use the Darboux representation (\ref{eqDarbouxRep}) to construct examples of closed curves of constant torsion in $\R^3$.
The basis of our construction is the introduction of a two-parameter family of spherical curves $B = (h, k, l)$, where $h$, $k$, and $l$ are certain trigonometric polynomials satisfying $h^2 + k^2 + l^2 = 1$.
Closed curves of constant torsion are then obtained by finding parameter values that make the integrals in (\ref{eqDarbouxRep}) vanish.

\subsection{A two-parameter family of spherical epicycles}

We begin with an illustrative example for a planar curve to motivate our approach.
By Stokes' theorem, the planar curve 
\begin{equation}\label{eqPlaneCurve}
\rho(t) = (x(t), y(t)) = (a\cos t + b\cos(-2t), a\sin t + b \sin(-2t) )
\end{equation}
encloses the oriented area
$$
\frac12\int_0^{2\pi}(x\,y' -y\,x')dt
	= \frac12\int_{\rho_*} x\,dy -y\,dx = \pi(2 b^2 - a^2),
$$
where $\rho_*$ denotes the trace of $\rho$ and is viewed as an oriented 1-cycle in $\R^2$.
This area integral clearly vanishes when $b = a/\sqrt{2}$. 
What is interesting here is that the curvature of $\rho$ never vanishes.
Kinematically, we may think of $\rho$ as describing the clockwise circular motion of a point about the counterclockwise circular motion of another point, where $b$ and $a$ are the corresponding radii. 
Geometrically, the integral vanishes because $\rho_*$ decomposes into four simple closed cycles enclosing oriented areas that cancel out when $b = a/\sqrt{2}$.
The fact that the curvature is not zero is related to the fact that this curve has winding number $3$, i.e., $\rho$ does not possess a regular isotopy to the identity map of the unit circle $S^1$.
On the other hand, if a curve possesses such an isotopy and encloses vanishing area, it must intersect itself in such a way as to produce inflection points, i.e., the corresponding curvature function is signed.

We now wish to adapt the kinematic interpretation of (\ref{eqPlaneCurve}) to the unit sphere.
To this end, we choose the vector $U = (1/\sqrt{3},1/\sqrt{3},1/\sqrt{3})^\mathrm{t}$ as our central point, and we seek a curve whose trace is invariant under a rotation of $2 \pi / 3$ about the axis defined by this point.
This will ensure that
\begin{equation}\label{eqSimpleIntegrals}
\int x\, dy-y\,dx = \int y\, dz -z\, dy = \int z\, dx - x\, dz,
\end{equation}
so we will only need to find parameter values that make one integral (and hence all three integrals) vanish.

In what follows, we use $(F)_k$ to denote the $k$-th column of a frame $F$, and we extend $U$ to an orthonormal frame $C$ by setting $V= (-1/\sqrt{6}, -1/\sqrt{6},2/\sqrt{6})^\mathrm {t}$, $W= U\times V = (1/\sqrt{2}, -1/\sqrt{2},0)^\mathrm{t}$, and $C = (U,V,W)$.
We will refer to $C$ as the \emph{central frame}.
Now define $R_t$ and $S_t$ by
$$
R_t = \begin{pmatrix}
	1 & 0 & 0 \\
	0 & \cos t & -\sin t \\
	0 & \sin t & \cos t 
\end{pmatrix}, \qquad
S_t = \begin{pmatrix}
	\cos t & -\sin t & 0 \\
	\sin t & \cos t & 0 \\
	0 & 0 & 1
\end{pmatrix}.
$$
It follows that the right multiplications $F \rightarrow F\,R_t$ and $F \rightarrow F\,S_t$ rotate $F$ about its first and third column, respectively, by an angle $t$ in the right handed sense.
We further define the orthogonal matrix $Q_t$ by the formula (cf. \cite{cushman-bates} for example)
$$
Q_t = I + \sin t\, \bar{U} + (1-\cos t)\,\bar{U}^2,
$$
where
$$
\bar{U} = \begin{pmatrix} 0 & -1/\sqrt{3} & 1/\sqrt{3} \\
                         1/\sqrt{3} & 0 &  -1/\sqrt{3} \\
                         -1/\sqrt{3} & 1/\sqrt{3} & 0 
          \end{pmatrix}
$$
is the antisymmetric matrix associated to the vector $U$.
We then have that the left multiplication $C \rightarrow Q_t\,C$ rotates $C$ about $U$ by an angle $t$.
\begin{definition}
We define the spherical $(m, n)$-epicycle with radii $(\alpha, \beta)$ to be the curve $B$ given by the rule
\begin{equation}\label{eqSphericalCurve}
B(t) = (Q_{m\,t}\,C\,S_\alpha\,R_{n\,t}\,S_\beta)_1.
\end{equation}
\end{definition}
To motivate this definition, note first that $(C\,S_\alpha)_1$ is the vector $U$ rotated about $W$ through an angle $\alpha$.
It follows, for $m \neq 0$, that $(Q_{m\,t}\,C\,S_\alpha)_1$ traces out a circle of geodesic radius $\alpha$ and center $U$ on the sphere $m$ times for $0 \leq t \leq 2\pi$.
To understand the epicyclic behavior, let us first define $U(t,\alpha) = (Q_{m\,t}\,C\,S_\alpha)_1$, $V(t,\alpha) = (Q_{m\,t}\,C\,S_\alpha)_2$, and $W(t,\alpha) = (Q_{m\,t}\,C\,S_\alpha)_3$.
From our previous definitions, it is clear that right multiplication by $R_{n\,t}$ rotates the plane orthogonal to $U(t,\alpha)$ through an angle of $n\,t$.
Let $\widehat{W}(t,\alpha)$ denote the image of $W(t,\alpha)$ under this rotation.
Now, right multiplication by $S_\beta$ rotates the plane orthogonal to $\widehat{W}(t,\alpha)$, moving $U(t,\alpha)$ by an angle $\beta$.
Thus $B$ moves about a circle of geodesic radius $\beta$ the center of which traces out the circle $U(t, \alpha)$.

Note that, for any choice of $(m, n)$, we have a family of curves that depend on the two parameters $(\alpha, \beta)$.
These curves will have the desired symmetry property if $n$ is a multiple of $3$.
In this case, we expect to find a one-parameter family of epicycles
for which the integrals (\ref{eqSimpleIntegrals}) vanish.

\subsection{Two examples of closed curves of constant torsion}

We will now discuss two examples of interest and use them as the basis for a conjecture.
Note that, since the components of $B$ in (\ref{eqSphericalCurve}) are trigonometric polynomials, the integral equation relating $\alpha$ and $\beta$
can be evaluated, but the result is unwieldy, in general.
We therefore only state formulas for the special cases of $\alpha$ and $\beta$ that are of interest.
Similarly, the curve $\gamma$ given by (\ref{eqIntegralRep}) can be exactly determined, but this formula is too complex to merit presentation.
The torsion of $\gamma$ in both examples is $\tau = 1$.  Positive torsion means that the curves twist from left to right, and are said to be {\it dextrorse}.  It is then trivial to produce similar examples with torsion $\tau = -1$, which twist form right to left, and are said to be {\it sinistrorse}.

For our first example, we set $m = 1$ and $n = -3$ in (\ref{eqSphericalCurve}).
This produces a spherical analog to (\ref{eqPlaneCurve}), and the corresponding projected area integral for $B$ is given by
$$
\frac{\sqrt{3}}{\pi}\int_0^{2\pi}(x\,y' -y\,x')dt
	= 2\sin^2\alpha\,\cos^2\beta
	+ \left(\cos^2\alpha -6\cos\alpha + 1 \right) \sin^2\beta.
$$
When $\alpha = \pi/4$, we find that this integral vanishes if
$$
\beta = \frac{1}{2}\cos^{-1}\left(\frac{67 - 24 \sqrt{2}}{71}\right).
$$
These parameter values produce the spherical curve of positive geodesic curvature shown in Figure \ref{gSPhericalFloret}.
By virtue of Proposition \ref{pKappaBgKappa}, the corresponding curve of constant torsion, shown in Figure \ref{gThreeLobedUnknottedPositiveK}, has everywhere positive curvature.

\begin{figure}[htb]
\setlength\fboxsep{0pt}
\setlength\fboxrule{0pt}
\fbox{
    \includegraphics[trim = 30mm 30mm 30mm 30mm, clip, scale=0.75]
        {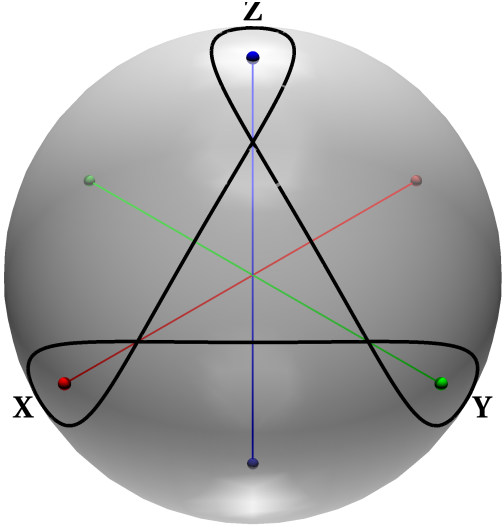}
}
\caption{
A spherical $(1, -3)$-epicycle with everywhere positive geodesic curvature.
}\label{gSPhericalFloret}
\end{figure}

\begin{figure}[htb]
\setlength\fboxsep{0pt}
\setlength\fboxrule{0pt}
\fbox{
    \includegraphics[trim = 0mm 30mm 20mm 50mm, clip, scale=0.5]
        {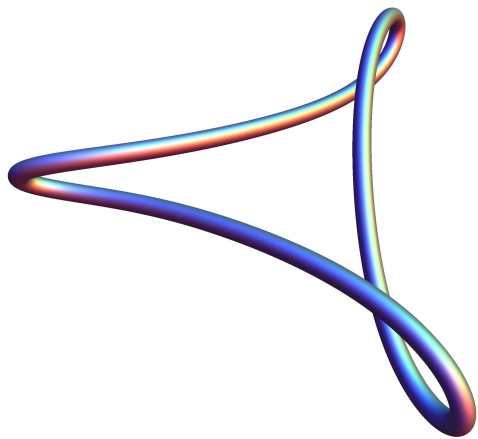}
}
\caption{
The unknotted curve of constant torsion and everywhere positive curvature corresponding to the curve in Figure \ref{gSPhericalFloret}.
}\label{gThreeLobedUnknottedPositiveK}
\end{figure}

Next, we consider the case when $m = 2$ and $n = -3$, for which  we obtain the area integral
$$
\frac{\sqrt{3}}{2\pi}\int_0^{2\pi}(x\,y' -y\,x')dt
	= 2\sin^2\alpha\,\cos^2\beta
	+ \left(1 - 3\cos\alpha +  \cos^2\alpha  \right) \sin^2\beta.
$$
When $\alpha = \pi/3$, we find that this integral vanishes if
$$
\beta = \frac{1}{2}\cos^{-1}\left(-\frac{5}{7}\right).
$$
This case produces a spherical curve of \emph{signed geodesic curvature}, which is shown in Figure \ref{gSPhericalFloret-p2m3}.
Here, the definition of signed geodesic curvature is the obvious analog for curves lying in surfaces of the definition of signed curvature given in Section \ref{scIntroduction}.
The corresponding curve of constant torsion is shown in Figure \ref{gTrefoilKnot}.

\begin{figure}[htb]
\setlength\fboxsep{0pt}
\setlength\fboxrule{0pt}
\fbox{
    \includegraphics[trim = 0mm 0mm 0mm 0mm, clip, scale=0.5]
        {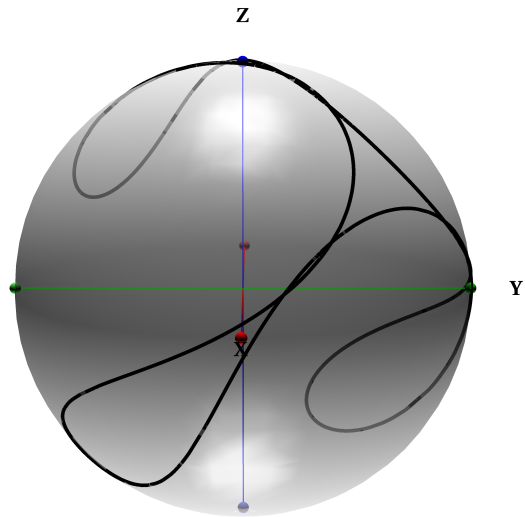}
}
\caption{
A spherical $(2, -3)$-epicycle with signed geodesic curvature.
This view and the symmetry of $B$ shows that it has six inflection points.
}\label{gSPhericalFloret-p2m3}
\end{figure}

\begin{figure}[htb]
\setlength\fboxsep{0pt}
\setlength\fboxrule{0pt}
\fbox{
    \includegraphics[trim = 0mm 30mm 20mm 50mm, clip, scale=0.5]
        {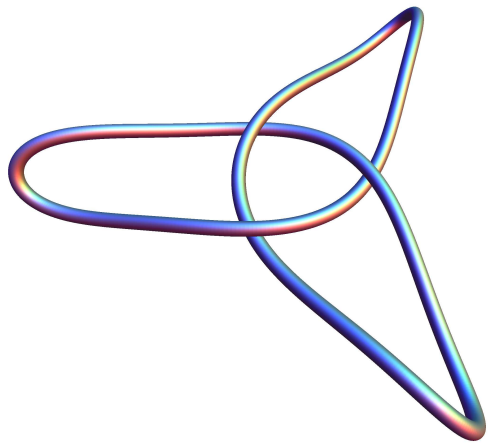}
}
\caption{
The trefoil knot of constant torsion and signed curvature corresponding to the curve in Figure \ref{gSPhericalFloret-p2m3}.
}\label{gTrefoilKnot}
\end{figure}

A comparison of Figures \ref{gSPhericalFloret-p2m3} and \ref{gTrefoilKnot} shows a clear correspondence between the inflection points of $B$ in Figure \ref{gSPhericalFloret-p2m3} and the positions of the points in Figure \ref{gTrefoilKnot} where the curve $\gamma$ passes through its rectifying plane.
The latter lead to the crossings that one requires in a knotted curve.
Hence, these pictures suggest the following.
\begin{conjecture}\label{cNoPositiveKnots}
Closed, knotted curves of constant torsion must have signed curvature.
\end{conjecture}

\section{Curves of constant torsion on ovaloids}\label{scOvaloids}

We define an \emph{ovaloid} to be a surface in $\R^3$ that is strictly convex and geodesically complete.
Such a surface is characterized by having the topology of a sphere and strictly positive Gaussian curvature.
In this section, we first give an explicit characterization of curves of constant torsion on the sphere and conclude that closed curves of this type do not exist.
We then discuss the problem of extending this nonexistence result to the general ovaloid.

\subsection{Known results for spherical curves}

It has long been known that the curvature $\kappa$ and torsion $\tau$ of a spherical curve are related by the differential equation
\begin{equation}\label{eqKappaTauODE}
\frac{\tau}{\kappa} - \left(\frac{\kappa^{\prime}}{\kappa^2\tau}\right)^{\prime} = 0.
\end{equation}
Here the differentiation is with respect to the arclength parameter $s$.
One may find this result in relatively recent books on differential geometry, such as \cite{goetz, guggenheimer}, as well as \cite{eisenhart}.
From this follow several interesting results, such as the fact that a closed spherical curve $\gamma$ satisfies the integral relations 
\begin{equation}\label{eqIntegralIdentities}
I_n = \int_\gamma \kappa^n\tau \,\dee s = 0, \quad n= 0,\pm 1,\pm 2,\dots
\end{equation}
This is proven in {\sc Saban} \cite{saban}, who goes on to show a sort of converse: if the integral $I_n$ vanishes for a fixed but arbitrary integer $n \geq 0$ and for all closed curves $\gamma$ on a surface $M$, then $M$ is necessarily either a plane or a sphere.
Breuer and Gottlieb \cite{breuer-gottlieb} were able to integrate (\ref{eqKappaTauODE}), and Wong \cite{wong} formulated their result as follows.
\begin{theorem}
A $C^4$ curve $\gamma(s)$, parametrized by its arclength $s$, with curvature $\kappa$ and torsion $\tau$ is a spherical curve if and only if 
\begin{equation}\label{eqWong} 
\left(A \cos\int_0^s \tau \,\dee \sigma + B\sin\int_0^s \tau\,\dee \sigma\right)\kappa(s) = 1,
\end{equation}
where $A$ and $B$ are constants.  Moreover, the curve lies on a sphere of radius $\sqrt{A^2+B^2}.$
\end{theorem}

\subsection{The case of constant torsion on the sphere}

In what follows we will take the torsion $\tau$ of the spherical curve $\gamma$ to be constant and the radius of the sphere to be one, centered at the origin.
We also fix the position $\gamma(0)$ of the zero of the arclength parameter so that the curvature of $\gamma$ is seen from (\ref{eqWong}) to be given by
$$ 
\kappa(s) = \sec (\tau s).
$$
The maximum arclength that $\gamma$ can have in this case is $\pi/|\tau|$, and $\kappa(s) \rightarrow \pm \infty$ as $s \rightarrow \pm \pi/(2\,|\tau|)$.
It follows that there can be no smooth, regular closed curve of constant torsion since the curvature function of such a curve would necessarily be smooth and periodic.

We would now like to give an explicit parametrization for curves of constant torsion on the sphere.
One approach would be to integrate the Frenet equations (\ref{eqFrenetMatrixForm}), which would require two quadratures.
We can eliminate one quadrature by using another moving frame adapted to the surface of the sphere.
More precisely, consider the \emph{Darboux frame} $\hat{F} = ({\bf t, u, \nu})$, where ${\bf t} = T$, $\nu$ is the surface normal restricted to $\gamma$, and ${\bf u} = \nu\times {\bf t}$ is the tangent normal of the curve (cf. \cite{guggenheimer, spivak3}).
Since $\nu =  \gamma$ on the unit sphere, we obtain $\gamma$ with one quadrature by integrating the following differential equation.
\begin{equation}\label{eqSphericalDarboux}
\hat{F}^{\prime} = \hat{F}\,\hat{A}, \qquad
\hat{A} =
	\begin{pmatrix}
		0			&	-\kappa_g	&	-1 \\
		\kappa_g 	&	0			&	0\\
		1			&	0			&	0
	\end{pmatrix}
\end{equation}
Here $\kappa_g$ denotes the geodesic curvature, which in our case is $\kappa_g=\tan (\tau s)$.
To integrate this equation, we first observe that, since the flow preserves the length of the rows or columns of $\hat{F}$, we may restrict our view to the unit sphere.
Following {\sc darboux} \cite{darboux} and {\sc eisenhart} \cite{eisenhart}, we stereographically project the sphere onto the plane and find that (\ref{eqSphericalDarboux}) is transformed into the Riccati equation
$$
\frac{d\theta}{ds} = \frac{i}{2} - i \kappa_g \theta -\frac{i}{2}\theta^2.
$$
Setting $\theta(s) = \psi(s) -\tan (\tau s)$, we find that $\psi$ satisfies
$$
\psi^{\prime} = \left(\frac{i}{2} + \tau \right) \sec^2 (\tau s) - \frac{i}{2} \psi^2.
$$
We now change the independent variable by $s=2u/i$ and find that 
$$
\frac{d\psi}{du} = (1-2i\tau) \sec^2(2i\tau u) - \psi^2.
$$
The substitution $\psi = \phi_u/\phi $ then produces 
$$ 
\phi^{\prime\prime} = (1-2i\tau)\sec^2(2i\tau u) \, \phi,
$$
and another change of independent variable $a = \tan(2i \tau u) $ yields
$$
(1+a^2)\frac{d^2\phi}{da^2} + 2a\frac{d\phi}{da} + \frac{1-2i\tau}{4\tau^2} \phi = 0.
$$
The solution to this equation is a linear combination of Legendre functions\footnote{Recall that the general solution to the equation $(1+x^2)y^{\prime\prime}+2xy^{\prime} +ky=0$ for constant $k=\lambda(\lambda+1)$ and $y=y(x)$ is $y =  c_1P_{\lambda}(
(\sqrt{1-4k}-1)/2, ix) + c_2 Q_\lambda( (\sqrt{1-4k}-1)/2, ix) $.}.
This allows us to find an explicit closed form solution to the original differential equation for a spherical curve of constant torsion.
An example of a solution curve, rendered in (heavy) black, is shown in Figure \ref{gSphericalCurve}.

It is perhaps surprising that we have replaced a real linear differential equation of order three by a complex differential equation of order two, without ever really seeming to integrate anything.
The explanation for this is as follows: the third order real differential equation describes a curve in the rotation group for the standard action of $\SO(3)$ on $\R^3$.
However, we may realize this action in another way: the linear action of $\U(2)$ on $\C^2$ preserves the unit sphere $S^3\in \C^2$, and the action of the subgroup $\SU(2)$ passes to the quotient $S^2$ of the sphere $S^3$ by $\U(1)$ (the Hopf map). 
This action passes to the quotient $\SU(2)/\pm = \PSU(2)\approx \SO(3)$ and is realized as transformations on the complex plane by linear fractional transformations.

Perhaps of more interest is that we now can (in principle) derive an analytic expression for the central angle $\zeta(\tau)$ between the two limit points of a curve with torsion $\tau$.
(The formula would be complex.)
It is of interest to note that the limiting value of $\zeta$ as $\tau\rightarrow 0$ is $\pi$.
A plot of the central angle versus the torsion is shown in Figure \ref{gZetaTPlot}.
Note that the orange (light) curve in Figure \ref{gSphericalCurve} is the trace of the limit points of $\gamma$ as $\tau$ varies.
The initial position and frame for $\gamma$ are the same for all values of $\tau$.

\begin{figure}[htb]
\setlength\fboxsep{0pt}
\setlength\fboxrule{0pt}
\fbox{
    \includegraphics[trim = 30mm 30mm 30mm 30mm, clip, scale=0.75]
        {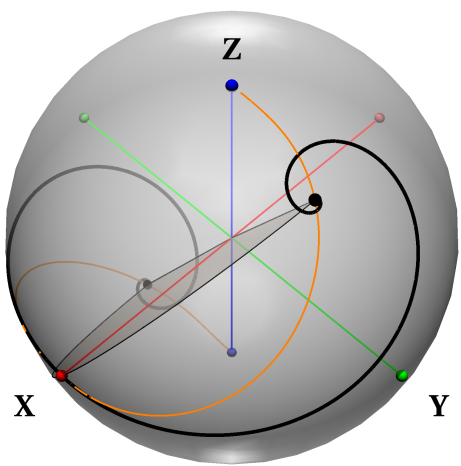}
}
\caption{
A curve of constant torsion \( \tau = 1/2 \) on the unit sphere.
}\label{gSphericalCurve}
\end{figure}

\begin{figure}[htb]
\setlength\fboxsep{0pt}
\setlength\fboxrule{0pt}
\fbox{
    \includegraphics[trim = 0mm 0mm 0mm 0mm, clip, scale=0.5]
        {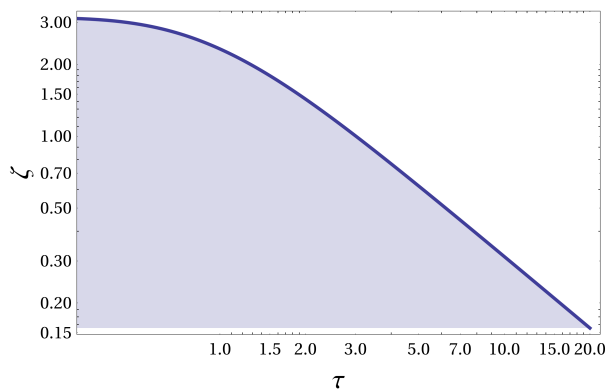}
}
\caption{
This log-log plot of \( \zeta \) versus \( \tau \) suggests that
 \( \zeta(\tau) \approx c / \tau \) for large \( \tau \) and for some
positive constant \( c \).
Note also that $\zeta \rightarrow \pi$ as $\tau \rightarrow 0$.
}\label{gZetaTPlot}
\end{figure}

\subsection{Closed curves of constant torsion on ovaloids}

First, we would like to point out that curves of constant torsion exist at least locally on arbitrary surfaces.
They arise as solutions to the Darboux equations discussed below with the added condition (\ref{eqTau_g}).
The result is a system of algebraic differential equations that will be discussed more fully in \cite{bates-melko}. 

We have seen that there is no closed curve of nonzero constant torsion on the sphere.
There are extensions of this result in a couple of different directions.
As a first example, there is the following.
\begin{theorem}\label{tSimpleCurve}
A simple closed curve that lies on the boundary of its convex hull, and without points of zero curvature, has at least four points where the torsion vanishes.
\end{theorem}

This theorem may be found in \cite{bisztriczky, nuno_ballesteros-romero_fuster, sedykh94}.  There are even versions with singular points \cite{romero_fuster-sedykh95}.  In particular, these imply that there is no simple closed curve of constant nonzero torsion on the surface of an ovaloid.  However, it seems that it is still unknown whether or not one can find a closed curve of nonzero constant torsion on the surface of an ovaloid if the curve is not simple.  We formalize this as

\begin{conjecture}
There are no closed curves of nonzero constant torsion on an ovaloid.
\end{conjecture}

In support of this conjecture, if there are such curves, their torsion can not be too large, as we have the upper bound given in Theorem \ref{tTauBound}.
Before proving this theorem, we recall some fundamental properties of curves on surfaces.
Details may be found in \cite{spivak3}.

For the sake of simplicity, we suppose that $M$ is a surface with only isolated umbilics and that $(\xi_1, \xi_2)$ are unit vector fields tangent to the principal directions in $M$.
For a curve $\gamma$ in $M$, we define $\vartheta$ to be the directed angle from $\xi_1$ to the unit tangent $T$ of $\gamma$, and we define $\varphi$ to be the directed angle from the principal normal $N$ of $\gamma$ to the surface normal $\nu$.
As in the spherical case, we define the Darboux frame to be $\hat{F} = ({\bf t}, {\bf u}, \nu)$, where ${\bf t} = T$, $\nu$ is the surface normal restricted to $\gamma$, and ${\bf u} = \nu \times {\bf t}$ is the tangent normal of the curve.
In this general context, the Darboux equations become
$$
\gamma^\prime = {\bf t}, \qquad
\hat{F}^{\prime} = \hat{F}\,\hat{A}, \qquad
\hat{A} =
	\begin{pmatrix} 0 & -\kappa_g & -\kappa_n \\
		\kappa_g & 0 &  -\tau_g\\
		\kappa_n & \tau_g &  0
	\end{pmatrix},
$$
where $\kappa_g$, $\kappa_n$, and $\tau_g$ are, respectively, the geodesic curvature, the normal curvature, and the geodesic torsion of $\gamma$.
Note that the Frenet and Darboux frames are related by the change of gauge that rotates the normal plane of $\gamma$ through an angle of $\varphi - \pi / 2$, taking $N$ to ${\bf u}$ and $B$ to $\nu$.
The geodesic torsion $\tau_g$ is related to the principal curvatures $k_1$ and $k_2$ and the torsion $\tau$ by the equations
\begin{equation}\label{eqTau_g}
\tau_g 
	=  \tau + \varphi^{\prime}
	= \frac{1}{2} (\kappa_2 -\kappa_1) \sin\vartheta.
\end{equation}

\begin{theorem}\label{tTauBound}
Let $\mu$ be the maximum difference of the two principal curvatures on a compact surface $M$.
Then a closed curve $\gamma$ of constant torsion $\tau$ on $M$ satisfies
\begin{equation*}\label{eqTauMu}
 |\tau| < \frac{\mu}{2}.
\end{equation*}
\end{theorem}

\begin{proof}
By virtue of (\ref{eqTau_g}), we have
\begin{equation*}
\varphi^{\prime}
	= \frac{1}{2} (\kappa_2 -\kappa_1) \sin\vartheta - \tau .
\end{equation*}
Hence, if $\tau \geq \mu / 2$, we must have that $\varphi^{\prime} \leq 0$ and therefore that $\varphi$ is a monotonically decreasing function.
Similarly, if $\tau \leq - \mu / 2$, then $\varphi$ must be increasing.
However, this is not possible, since $\varphi$ must be periodic.   
\end{proof}

Suppose, now, that $M$ is an ovaloid.
An immediate consequence of Theorem \ref{tTauBound} is that, if $M$ is nearly spherical, $\mu$ will be small, and hence any closed curve of nonzero constant torsion must be nearly planar.
However, Theorem \ref{tSimpleCurve} makes it seem unlikely that such curves exist.
In the limiting case when $\mu \rightarrow 0$, $M$ becomes a sphere, and the corresponding curves would become circles.

In addition to this, there are infinitely many integral constraints that a closed curve of constant torsion on any surface must satisfy.
First of all, integrating (\ref{eqTau_g}) once around the curve yields
$$
\int_\gamma \tau_g \, ds = \tau L,
$$
where $L$ is the arclength of the curve.
Furthermore, suppose that  $\delta$ is the spherical image of the curve $\gamma$ under the Gauss map of the surface (i.e. $\delta = \nu \circ \gamma$).
Then, by (\ref{eqIntegralIdentities}), we have
$$
\int_\delta k_{\delta}^n\tau_{\delta}\,d\bar{s} = 0,\qquad
	n = 0,\pm1,\pm2,\dots
$$
where $k_{\delta}$ and $\tau_{\delta}$ denote the curvature and torsion of $\delta$, respectively, and $\bar{s}$ is the arclength parameter of $\delta$.
We can pull all of these constraints back via the map $\nu^*$ to the surface $M$ to realize them as constraints on our curve $\gamma$.
It seems plausible that these constraints will force $\tau$ to be zero if we add the condition that $M$ is an ovaloid, but we have not yet managed to do this.
Finally, we point out that we do believe closed curves of nonzero constant torsion exist on surfaces with regions of negative Gaussian curvature, such as the standard torus in $\R^3$, and we hope to provide examples in our forthcoming paper \cite{bates-melko}.

\section{Notes}

As we were finishing this work, two articles came to our attention that are
closely related to this paper.

\begin{enumerate}
\item In a brief note in 1977, {\sc weiner} \cite{weiner} has a similar
discussion of a closed curve of constant torsion with positive curvature, but he does not provide a picture.
We feel that our more general approach is worthy of consideration, and we point out that Conjecture \ref{cNoPositiveKnots}, to the best of our knowledge, is new.
What is remarkable here is that we were both inspired by the identical planar curve!

\item A preprint by {\sc kazaras} and {\sc sterling} \cite{kazaras-sterling} 
gives a different characterization for a spherical curve of constant torsion than ours.
The difference here is that we wind up with rational functions of Legendre
functions, and they wind up with power series of hypergeometric functions.  
\end{enumerate}

\bibliography{curves_of_constant_torsion.bbl}

\bibliographystyle{plain}

\vspace{20pt}

\begin{tabular}{@{}p{2.5in}p{4in}}
Larry M. Bates						&	O. Michael Melko\\
Department of Mathematics           &	Dorian Apartments \# 109\\	
University of Calgary               &	19 7th Ave. S.E.\\
Calgary, Alberta					&	Aberdeen, SD 57401\\
Canada T2N 1N4      				&   U.S.A.\\
bates@ucalgary.ca                   &   mike.melko@gmail.com\\
\end{tabular}

\end{document}